\documentclass[12pt]{amsart}

\usepackage[cp1251]{inputenc}
\usepackage[english]{babel}
\usepackage{amsmath,amssymb,amsthm,amsfonts, enumerate}
\usepackage{a4wide}
\newtheorem{thm}{Theorem}
\newtheorem{prop}{Proposition}
\newtheorem{lem}{Lemma}

\theoremstyle{definition}
\newtheorem*{defn}{Definition}
\newtheorem{cor}{Corollary}
\newtheorem*{corr}{Corollary}

\newcommand{\T}[1]{T^{(d)}_{#1}}

\DeclareMathOperator{\cpa}{ConPAut}
 \DeclareMathOperator{\pa}{PAut}
\DeclareMathOperator{\dom}{dom}
 \DeclareMathOperator{\ran}{ran}
  
\DeclareMathOperator{\aut}{Aut}
\newcommand{\isd}{\mathcal{IS}_d}
\newcommand{\nd}{\mathcal{N}_d}
\newcommand{\espn}[1]{S_{#1}^{P \mathcal{N}_d}}
\newcommand{\spn}{S^{P \mathcal{N}_d}}

\newcommand{\pwr}{\mathop\wr_p}
\newcommand{\ipwr}[1]{\mathop\wr_p^{#1} \isd}
\newcommand{\wisd}{\isd \pwr \isd}

\newcommand{\lc}{\mathbin\mathcal{L}}
\newcommand{\rc}{\mathbin\mathcal{R}}
\newcommand{\hc}{\mathbin\mathcal{H}}
\newcommand{\dc}{\mathbin\mathcal{D}}
\newcommand{\jc}{\mathbin\mathcal{J}}

\newcommand{\g}{\Gamma}
\newcommand{\sbs}{\subset}
\newcommand{\sbeq}{\subseteq}
\newcommand{\ra}{\rightarrow}
\newcommand{\vph}{\varphi}
\newcommand{\s}{\sigma}
\newcommand{\ld}{\{1, \ldots, d\}}
\input qobitree
\begin{document}

\title{Combinatorics of partial wreath power of finite
inverse symmetric semigroup $\mathcal{IS}_d$}
\author{Eugenia Kochubinska}

\begin{abstract}
	We study some combinatorial properties of partial wreath $k$-th power of the semigroup $\mathcal{IS}_d$.
	In particular, we calculate its order, the number of idempotents
	and the number of D-classes.
\end{abstract}
\maketitle

\section{Introduction}
The wreath product of semigroups has appeared as a generalization
to semigroups of the corresponding construction for groups.
Firstly transformation wreath product of transformation semigroups
has appeared as a natural generalization of the wreath product of
permutation groups (cf. \cite{Meldrum}). Later different
modifications have been introduced, for instance, partial wreath
product of arbitrary semigroup and semigroup of partial
transformation was defined in \cite{Petr} and construction related
to this one, namely inverse wreath product of inverse semigroups,
was proposed in \cite{Hou}. Wreath products provide a means to
construct a semigroup with certain properties. They also appear in
certain natural settings, that allows to lighten the study of
known semigroups presenting them if possible as a wreath product
of appropriate semigroups

The article discusses the partial wreath product of two finite
symmetric semigroup $\isd$ and a generalization of this
construction to the case of more then two factors. It is proved
that the partial wreath $k$-th power of the semigroup $\isd$ is
isomorphic to the appropriate subsemigroup  of  semigroup of
partial automorphisms of the rooted $k$-level $d$-regular tree. We
study some combinatorial properties of $\wr_p^k \mathcal{IS}_d$,
in particular, we calculate its order and the number of
idempotents and the number of $\mathcal D$-classes. Also, we
describe Green's relations of the partial wreath power of $\isd$
and calculate the number of $\mathcal D$-classes, the number of
elements in a given $\mathcal D$-class and the number of $\mathcal
R$- and $\mathcal L$-classes in this $\mathcal D$-class.
\section{The partial wreath power of semigroup $\isd$}
Let $\nd=\ld$. Define  $\spn$ by
\begin{equation*}
\spn=\{f:\nd \ra \isd | \dom(f) \sbeq \nd\}.
\end{equation*}
as the set of functions from subsets of $\nd$  to $\isd$. If $f,g
\in \spn$, we define the product $fg$ by:
\begin{equation*}
\dom(fg)=\dom(f)\cap\dom(g), (fg)(x)=f(x)g(x) \text{for all } x\in
\dom(fg).
\end{equation*}
If $a\in \isd, f\in \spn$, we define $f^a$ by:
\begin{equation*}
\begin{gathered}
\dom(f^a)=\{x \in \dom(a); xa\in
\dom(f)\}=(\ran(a)\cap\dom(f))a^{-1}\\
(f^a)(x)=f(xa).
\end{gathered}
\end{equation*}
\begin{defn}
The partial wreath square of semigroup $\isd$ is defined as the
set $\{(f,a)\in \spn \times \isd\,|\,\dom(f)=\dom(a) \}$ with
composition defined by
$$(f,a)\cdot (g,b)=(fg^a,ab)$$
Denote it by $\isd \pwr \isd$.
\end{defn}
The partial wreath square of $\isd$  is a semigroup, moreover, it
is an inverse semigroup \cite[Lemmas 2.22 and 4.6]{Meldrum}. We
may recursively define any  partial wreath power of the finite
inverse symmetric semigroup.
\begin{defn}
The partial wreath $k$-th power of semigroup $\isd$ is defined as
semigroup $\ipwr{k}=\big( \ipwr{k-1} \big) \pwr \isd =\{(f,a) \sbs
\espn{k-1} \times \isd\,|\,\dom(f)=dom(a)\}$ with composition
defined by
$$(f,a) \cdot (g,b)=(fg^a,ab),$$ where $\espn{k-1}=\{f:\nd \ra
\pwr^{k-1}\isd,\,\dom(f) \sbeq \nd\}$, $\pwr^{k-1} \isd$ is the
partial wreath $(k-1)$-th power of semigroup $\isd$
\end{defn}
For an arbitrary function $F$ we denote
$F^k(x)=\underset{k}{\underbrace{F(F\ldots (F}}(x))\ldots)$.
\begin{prop}
$|\ipwr{k}|=S^k(1)$, where $S(x)=\sum_{i=1}^d \binom{n}{i} ^2
i!x^i$
\end{prop}
\begin{proof}
We provide the proof by induction on $k$.

Let $k=1$, then $|\isd|=\sum_{i=1}^d \binom{n}{i}^2i!=S(1)$ (cf.
\cite{GM}).

Assume that we know the order of the partial wreath $(k-1)$-th
power of semigroup $\isd$: $|\ipwr{k-1}|=S^{k-1}(1)$. Prove that
$|\ipwr{k}|=S^k(1)$. The elements of semigroup $\ipwr{k-1}$ are
pairs $(f,a)\in \espn{k-1} \times {\isd}$ with $\dom(f)=\dom(a)$.
Let $P_A=\{a\in \isd| \dom(a)=A\}$. Then the number of all such
pairs $(f,a)$ is equal to
$$\sum \limits_{A\sbs \nd} \Big\vert \ipwr {k-1} \Big\vert^{|A|} \cdot
\big|P_A\big|=\sum_{i=1}^d\Big|\ipwr{k-1}\Big|^i \binom{n}{i}^2
i!=S(\big|\ipwr{k-1}\big|)=S(S^{k-1}(1))=S^k(1).$$
\end{proof}
Let $E(\isd)$ be the set of idempotents  of semigroup $\isd$.
\begin{prop}
An element $(f,a) \in \wisd$ is an idempotent if and only if $a
\in E(\isd)$ and $f(\dom(a))\subseteq E(\isd)$.
\end{prop}

\begin{proof}
Let $(f,a)$ be idempotent, then $(f,a)(f,a)=(ff^a,a^2)=(f,a)$.
Hence, $ff^a=f$, $a^2=a$, i.e., $a\in \isd$ is an idempotent. It
follows from the equality $ff^a=f$ that for any $c \in \dom(a)$
$ff^a(ca)=f(ca)f^a(ca)=f(ca)f(ca^2)=f(ca)f(ca)$.

Conversely, let $(f,a)\in \ipwr{k}$ be such an element that $a \in
E(\isd)$ and $f(\dom(a))\subseteq E(\isd)$. Then for any $c \in
\dom(a)$ $f(ca)=f(ca)f(ca)$. So
$f(ca)=f(ca)f(ca)=f(ca)f(ca^2)=ff^a(ca)$. Since it holds for all
 $c \in \dom(a)$, we have
$(f,a)(f,a)=(f,a)$.
\end{proof}

Let $\T{k}$ be a rooted $k$-level $d$-regular tree. The partial
automorphism of the tree $\T{k}$ is such partial (i.e. not
necessarily completely defined) injective map $\varphi: V \T{k}
\rightarrow V \T{k}$ that subgraphs generated by domain of
$\varphi$ and range of $\varphi$ are isomorphic (i.e. $\varphi$
maps isomorphically certain subgraph of the tree $\T{k}$ on
another subgraph of the same tree). Partial automorphisms form a
semigroup under composition $ab(x)=b(a(x))$ , we will denote it by
$\pa \T{k}$. Evidently, this semigroup is an inverse semigroup.
Let $\cpa T$ be the semigroup of partial automorphisms of the tree
$T$,  defined on a connected graph containing root and preserving
the level of vertices. Further we will consider only partial
automorphisms of this type.
\begin{thm}\label{thm:paut_wr}
Let $\T{k}$ be a rooted $k$-level $d$-regular tree. Then
$$\cpa \T{k} \cong\ipwr{k}.$$
\end{thm}
\begin{proof}
We provide the proof by induction on $k$.

Let $\T{1}$ be one-level tree,  $\cpa \T{1}$ be the semigroup of
partial automorphisms of this tree defined as above. By
definition, $\cpa \T{1}$ contains partial automorphisms defined on
a connected subgraph and which don't move the root vertex and
preserve the level of vertices, then every partial automorphism
$\vph \in \cpa \T{1}$ is determined only by the vertices
permutation satisfying condition
$$\vph(i)=\left\{%
\begin{array}{ll}
  a_i , & \hbox{if $i \in \dom(\vph)$ ;} \\
  \emptyset, & \hbox{otherwise.} \\
\end{array}%
\right.$$ In other words, $\varphi$ is the partial permutation
from $\isd$. So, every partial automorphism $\varphi \in \cpa
\T{1}$ is uniquely defined by partial permutation $\s \in \isd$.
Thus, we have one-to-one correspondence between $\cpa \T{1}$ and
$\isd$. Hence $\cpa \T{1} \cong \isd.$

Assume that $\ipwr{k-1} \cong \cpa_{k-1}$.

Prove that $\ipwr{k} \cong \cpa_{k}$. Let $\vph \in \cpa_{k}$ and
$V_i$ be the $i$-th level of the tree $\T{k}$. Define a map $\psi:
\cpa_{k} \ra \ipwr{k}$ by: $\vph \mapsto
(\vph|_{T_{k-1}},\vph|_{V_1} )$, where $\vph|_{T_{k-1}}$ is a
partial automorphism that acts on the rooted subtrees, which root
vertices lie on the first level of the tree $\T{k}$ and belong to
$\dom(\vph|_{V_1})$. Hence $\vph|_{V_1} \in \isd$ and
$\vph|_{T_{k-1}}: \dom(\vph|_{V_1}) \ra \ipwr{k-1}$. Thus we may
establish correspondence between given partial automorphism $\vph
\in \cpa_{k}$ and a unique pair $(\s, f)$, where $\s \in \isd$,
$f:\nd \ra \ipwr{k-1},\dom(f)=\dom(\s)$. And we have $\cpa \T{k}
\cong\ipwr{k}$.

\end{proof}

\begin{prop}
Let $E(\ipwr{k})$ be the set of idempotents of semigroup
$\ipwr{k}$. Then $\big|E(\ipwr{k})\big|=F^k(1)=
\underset{k}{\underbrace{(((1+1)^d+1)^d \ldots+1)^d}}$, where
$F(x)=(x+1)^d$.
\end{prop}

\begin{proof}
It follows from the theorem  \ref{thm:paut_wr} that there exists
bijection between set of idempotents of semigroup $\ipwr{k}$ and
set of connected subgraphs of the tree $\T{k}$ with different
domains. We calculate number of idempotents as a number of such
subgraphs of the tree $\T{k}$, because idempotents of $\pa$ are
identity maps:  $id_{\g}: \g \ra \g$,  $\g \sbs \T{k}$.

We compute their number by induction on $k$. Let $k=1$, then
$\ipwr{1}=\isd$, consequently $|E(\ipwr{1})|=|E(\isd)|=2^d=F(1)$.

Assume that $|E(\ipwr{k-1})|=F^{k-1}(1)=|E(\pa \T{k-1})|$.

Find now the number of idempotents of semigroup $\pa \T{k}$. For
all $i=1, \ldots,d$ we can choose
 $i$-element subset   among the first level
 vertices in $\binom{d}{i}$ ways. Denote these subsets $A_i^j$, $i=1, \ldots,d$, $j=1, \ldots, \binom{d}{i}$
  Each vertex from $A_i^j$ is the root vertex of $(k-1)$-level tree. We know the
  number of idempotents of the semigroup $\pa \T{k-1}$, then

$|E(\ipwr{k})|=|E(\pa \T{k})|=
\sum_{i=1}^d\binom{d}{i}(F^{k-1}(1))^i=(F^{k-1}(1)+~1)^d=F(F^{k-1})=F^k(1)$.
\end{proof}
\section{Combinatorics of Green's relations}
\begin{thm}  \label{green}
Let $(f,a), (g,b) \in \ipwr{k}$. Then
\begin{enumerate} \rm
\item \label{it:itl} $(f,a)$ $\lc$ $(g,b)$ if and only if
$\ran(a)=\ran(b)$ and  $g^{b^{-1}}(z) \lc f^{a^{-1}}(z)$ for all
$z \in \ran(a)$ , where $a^{-1}$ is the inverse element for $a$;

\item \label{it:itr} $(f,a)$ $\rc$ $(g,b)$ if and only if
$\dom(a)=\dom(b)$ and $f(z) \rc g(z)$ for all $z \in \dom(a)$ ;

\item \label{it:ith} $(f,a)$ $\hc$ $(g,b)$ if and only if
$\ran(a)=\ran(b)$ and $\dom(a)=\dom(b)$,
$g^{b^{-1}}(z)~\lc~f^{a^{-1}}(z)$ and $f(z) \rc g(z)$ for $z \in
\dom(a) \cap \ran(a)$;

\item $(f,a)$ $\dc$ $(g,b)$ if and only if there exists a
bijection map $x:\dom(b) \ra \dom(a)$ such that $f(zx)$ $\dc$
$g(z)$.

\item $\dc$=$\jc$.
\end{enumerate}
\end{thm}
\begin{proof}
Green's relations  on semigroup $\isd$ are described in \cite{GM}.
\begin{enumerate}
\item Let $(f,a) \lc (g,b)$, then there exist $(u,x), (v,y) \in
\ipwr{k}$ such that $(u,x)(f,a)=(g,b)$ and $(v,y)(g,b)=(f,a)$,
i.e.
\begin{gather*}
(u,x)(f,a)=(uf^x,xa)=(g,b)\\
(v,y)(g,b)=(vg^y, yb)=(f,a)
\end{gather*}
We get from these equalities that $xa=b, yb=a$, and therefore $a
\lc b$ and then
 $\ran(a)=\ran(b)$, and also we get  $uf^x=g, vg^y=f$. Multiplying the
 both sides of the equality $xa=b$ by $a^{-1}$ from the
 left and by $b^{-1}$ from the right we obtain $b^{-1}x=a^{-1}$.
 Analogously we obtain  $a^{-1}y=b$.
Put $t=zb^{-1}$ for any $z\in \dom(b^{-1})=\ran(b)$, then
\begin{gather*}
uf^x(t)=g(t)\\
u(t)f(tx)=g(t)\\
u(zb^{-1})f(zb^{-1}x)=u(zb^{-1})f(za^{-1})=g(zb^{-1})\\
u(zb^{-1})f^{a^{-1}}(z)=g^{b^{-1}}(z)
\end{gather*}
Putting $t=za^{-1}$ for any $z \in \ran(a)=\ran(b)$, we
analogously get $v(za^{-1})g^{b^{-1}}(z)=f^{a^{-1}}(z)$. We have
$u(zb^{-1})f^{a^{-1}}(z)=g^{b^{-1}}(z)$ and
$v(za^{-1})g^{b^{-1}}(z)=f^{a^{-1}}(z)$. This implies
$f^{a^{-1}}(z) \lc g^{b^{-1}}(z)$ , $z\in \ran(a)=\ran(b)$.

Conversely, let $\ran(a)=\ran(b)$ and $f^{a^{-1}}(z) \lc
g^{b^{-1}}(z)$ $\forall \, z\in \ran(a)=\ran(b)$. From the first
condition we get  $a \lc b$, and hence there exist $x, y \in \isd$
such that $xa=b, yb=a$. From the second condition it follows that
there exist functions $u, v \in \espn{k}$ such that
$u(z)f^{a^{-1}}(z)=g^{b^{-1}}(z)$ and
$v(z)g^{b^{-1}}(z)=f^{a^{-1}}(z)$ , $z \in \ran(a)=\ran(b)$.
Consider $(u,x), (v,y) \in \ipwr{k}$, where $x,y,u,v$ are defined
as above. Then
$$(u,x)(f,a)=(uf^x,xa)=(uf^{ba^{-1}},b)=(g^{bb^{-1}},b)=(g,b)$$
and in the same way we get $(v,y)(g,b)=(f,a)$. Therefore $(f,a)
\lc (g,b)$.

\item  Let $(f,a)$ $\rc$ $(g,b)$, then there exist $(u,x), (v,y)
\in \ipwr{k}$ such that $(f,a)(u,x)=(g,b), \;(g,b)(v,y)=(f,a)$.
This is equivalent to  $ax=b, by=a, fu^a=g, gv^b=f$. This gives us
 the conditions $a$ $\rc$ $b$ , and hence $\dom(a)=\dom(b)$,  and
$fu^a=g, gv^b=f$. Consequently,  $f(z) \rc g(z) \; \forall \, z\in
\dom(a)$.

Conversely, let $(f,a), (g,b) \in \ipwr{k}$ and $\dom(a)=\dom(b)$,
$f(z) \rc g(z) \; \forall \, z\in \dom(a)$. From $\dom(a)=\dom(b)$
it follows $a \rc b$, then there exists $x,y \in \isd$ such that
$ax=b, by=a$, and from $f(z) \rc g(z) \; \forall \, z\in \dom(a)$
it follows that there exist $u', v' \in \espn{k-1}$ such that for
any $z \in \dom(a)$ $fu'(z)=g(z), gv'(z)=f(z)$. Define $u, v \in
\espn{k-1}$ by $u(za)=u'(z), \, v(zb)=v'(z)$. Then for $t \in
\dom(a)$ it holds $fu^a(t)=f(t)u(ta)=f(t)u'(t)=g(t)$ and
$gv^b(t)=f(t)$, then
\begin{gather*}
 (f,a)(u,x)=(fu^a, ax)=(g,b);\\
 (g,b)(v,y)=(f,a).
\end{gather*} Therefore, $(f,a) \rc (g,b)$.

\item As $\hc=\lc \wedge \rc$, this statement follows from the
first and second ones.

\item Let $(f,a) \dc (g,b)$. Then there exist $(h,c) \in \ipwr{k}$
such that $(f,a) \lc (h,c)$ and $(h,c) \rc (g,b)$. From $(f,a) \lc
(h,c)$ we get that $\ran(a)=\ran(c)$ and for $z\in \ran(a)$
$f^{a^{-1}}(z) \lc h^{c^{-1}}(z)$. Then there exist functions $u$
and $v$ such that $u(z)f^{a^{-1}}(z)=h^{c^{-1}}(z)$ and
$v(z)h^{c^{-1}}(z)=f^{a^{-1}}(z)$. Put $x=a^{-1}c$. By definition
of $\isd$ $x$ is a partial bijection map. We now obtain
 $f(zx) \lc h(z)$, and  $x: \dom(c)\ra \dom(a)$. From
$(h,c) \rc (g,b)$ we have that for $z \in \dom(b)$: $h(z) \rc
g(z)$ and $\dom(b)=\dom(c)$. From $\ran(a)=\ran(c)$ and
$\dom(b)=\dom(c)$ we get $|\dom(a)|=|\dom(b)|$.
 Thus there a exists bijection $x: \dom(b) \ra \dom(a)$ such that $f(zx) \dc h(z), z\in
\dom(b) \cap \ran(a)$.

Conversely, assume that 
there exists a  bijection map
 $x:\dom(b) \ra \dom(a)$ such that $f(zx) \dc g(z)$, i.e. there
 exists a function $h(z)$ such that $f(zx) \lc h(z)$ and $h(z) \rc
g(z)$. Let $u'(z), v'(z) \in \spn $ satisfy conditions
$u'(z)f(zx)=u'f^x(z)=h(z)$ and $v'(z)h(z)=f(zx)$. 
Put $c=xa$, then $c$ is partial bijection $c:\dom(b) \ra \ran(a)$
exists. Define $u(z)$ by $u(z)=u'(z)$ and $v(z)$ by
$v(z)=v'(zx^{-1})$. Then
\begin{gather*}
(u,x)(f,a)=(uf^x,xa)=(h,c), \\
(v,x^{-1})(h,c)=(f,a).\end{gather*} Hence $(f,a) \lc (h,c)$. As
$h(z) \rc g(z)$ and $\dom(c)=\dom(b)$, then $(h,c) \rc (g,b)$. It
implies $(f,a) \dc (g,b).$
 \item As $\ipwr{k}$ is finite  then $\dc =\jc$.
\end{enumerate}
\end{proof}

\begin{corr}
If $(f,a), (g,b) \in \wisd$, then
\begin{enumerate} \rm
 \item $(f,a)\lc(g,b)$ if and only if
$\ran(a)=\ran(b)$ and $\ran(g^{a^{-1}}(z))=\ran(f^{b^{-1}}(z))$
for all $z \in \ran(a)$;

\item $(f,a)$ $\rc$ $(g,b)$ if and only if $\dom(a)=\dom(b)$ and
$\dom(f(z))=\dom(g(z))$ for all $z \in \dom(a)$ ;

\item  $(f,a)$ $\hc$ $(g,b)$ if and only if $\ran(a)=\ran(b)$,
$\dom(a)=\dom(b)$, $\ran(g^{a^{-1}}(z))=\ran(f^{b^{-1}}(z))$ for
$z \in \ran(a)$,  and $\dom(f(z))=\dom(g(z))$ for $z \in \dom(a)$.
\end{enumerate}
\end{corr}

\begin{lem}\label{lem:dc}
Let $\s, \tau \in \pa \T{k}$. Then $\s \dc \tau$ if and only if
$\dom(\s) \cong \dom(\tau)$.
\end{lem}
\begin{proof}
Let $\s \dc \tau$, then there exists $\gamma \in \pa \T{k}$ such
that $\s \lc \gamma$ and $\gamma \rc \tau$. Thus,
$\ran(\s)=\ran(\gamma)$, $\dom(\gamma)=\dom(\tau)$. By definition
of semigroup $\pa \T{k}$ all these maps are isomorphisms between
their domains and ranges. It immediately follows that map
$\varphi=\gamma \s^{-1}: \dom(\tau) \ra \dom(\s)$ is isomorphism
from $\dom(\tau)$ to $\dom(\s)$, so $\dom(\s) \cong \dom(\tau)$.

Let $\dom(\s) \cong \dom(\tau)$. As before by definition of
semigroup $\pa \T{k}$ it follows $\dom(\s) \cong \ran(\s)$, hence
isomorphism $\gamma: \ran(\s) \ra \dom(\tau)$ exists. Therefore,
$\s \lc \gamma$ and $\gamma \rc \tau$. It implies $\s \dc \tau$.
\end{proof}
\begin{prop} \label{prop:dcl}
The number of $\dc$-classes of semigroup $\ipwr{k}$ equals
$P^k(1)$, where $P(x)=\binom{x+d}{d}$.
\end{prop}
\begin{proof}
By Theorem  \ref{thm:paut_wr} and Lemma \ref{lem:dc}the
calculation of the number of $\dc$-classes of semigroup $\ipwr{k}$
is equivalent to that of the number of non-isomorphic connected
subgraphs of the tree $\T{k}$ containing root vertex. Later on all
subgraphs are supposed to be connected and to contain root vertex.

Partition the set of all connected subgraphs of the tree $\T{k}$
into the classes of isomorphic subgraphs. Define the set of
graphs-representatives denoted by $GRep_k$ in the following way.

Consider firstly one-level $d$-regular tree. It is clear that the
set of all connected subgraphs is divided into $d+1$ class. We
choose a representative from each class and number them with
integers from 0 to $d$  in decreasing order of root vertices
degree. For example, if $d=3$ we have:

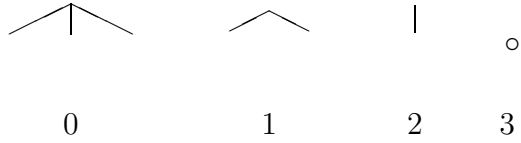
\begin{center}
\begin{tabular}{cccc}
\leaf{$\;\;\;\;$}\leaf{$\;\;\;\;$}\leaf{$\;\;\;\;$}\branch{3}{$\;\;\;\;$}\tree
&
\leaf{$\;\;\;\;\;\;$}\leaf{$\;\;\;\;\;\;$}\branch{2}{$\;\;\;\;$}\tree
& \leaf{$\;\;\;\;\;\;$}\branch{1}{$\;\;\;\;$}\tree &
\leaf{$\circ$} \tree
\\

\ & & &  \\
0 & 1 & 2 & 3
\end{tabular}
\end{center}
Define the following order relation on the set of
graphs-representatives $GRep_1$. Let $i_1, i_2$ be the numbers of
graphs $\g_1$ and $\g_2$ respectively. Then $\g_1>\g_2
\Leftrightarrow i_1 < i_2$.

Consider now 2-level tree $\T{2}$. Partition again the set of
connected subgraphs into classes of isomorphic subgraphs. Notice
that each vertex of  the first level of $\T{2}$ is a root vertex
of a one-level subgraph, which is isomorphic to a ceratin subgraph
from the set $GRep_1$. Attach a number sequence $(i_1, i_2,
\ldots, i_l)$ to each subgraph by, where $l$ is a degree of the
root vertex and $i_j$ is the number of subgraph from $GRep_1$
subgraph of $\T{2}$ with root vertex labelled  by $j$ is
isomorphic to. For example, the corresponding sequence for
subgraph

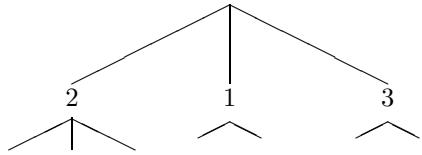
\begin{center}
\footnotesize%
\centering%
\leaf{$\;\;\;\;\;$}%
\leaf{$\;\;\;\;\;$}%
\leaf{$\;\;\;\;\;$}%
\branch{3}{2}%
\leaf{$\;\;\;\;\;$}%
\leaf{$\;\;\;\;\;$}%

\branch{2}{1}%
\leaf{$\;\;\;\;\;$}%
\leaf{$\;\;\;\;\;$}%

\branch{2}{3}%
\branch{3}{} \tree
\end{center}
is $(1,0,1)$. It is evident that connected subgraphs of $\T{2}$
are isomorphic if and only if corresponding sequences are equal up
to the permutation of sequences members. Choose a subgraph
described by non-decreasing corresponding sequence from each class
of isomorphic subgraphs. We call these subgraphs
graphs-representatives and define a linear order relation on the
set of graphs-representatives $GRep_2$ in the following way: let
$\g_1, \g_1 \in GRep_2$ and $a_1=(i_1, i_2, \ldots, i_m)$,
$a_2=(j_1, j_2, \ldots, j_n)$ be corresponding sequences. Then
$\g_1>\g_2$ if and only if $a_1<a_2$, set of sequences is
lexicographically ordered. For instance, if the number sequence
related to subgraph  $\g_1$ is $(0,0,0)$ and the number sequence
related to subgraph  $\g_2$ is $(0,0,1)$, then $\g_1>\g_2$. We
have linearly ordered set and we may arrange
graphs-representatives in decreasing order and number them in such
a way that 0 corresponds to the ``biggest'' graph.

Let now $\g_0>\g_1 > \ldots >\g_N$ be ordered set $GRep_{k-1}$ of
graphs-representatives of $(k-1)$-level tree $\T{k-1}$. Partition
again the set of all connected subgraph of the tree $\T{k}$ into
classes of isomorphic subgraphs. Attach again a number sequence
$(i_1, i_2, \ldots, i_l$) to each subgraph, where $i_j$ is the
number of corresponding graph from $GRep_{k-1}$, $i_j \in \{0, 1,
\ldots, N\}, j=\overline{1,l}, \; l \leq d$, and construct the set
of graphs-representatives $GRep_{k}$ of the $k$-level tree $\T{k}$
as above. It is easy to check that set $GRep_{k}$ has following
properties:
\begin{enumerate}
\item For all subgraph $\g \sbs \T{k}$ there exists subgraph
$\tilde{\g}$ from the set $ GRep_k$ such that $\g \cong
\tilde{\g}$; \item If $\g_1 \cong \g_2$, where $\g_1, \g_2 \sbs
GRep$ , then $\g_1=\g_2$.
\end{enumerate}
Therefore, we have to compute the cardinality of the set of
graphs-representatives $GRep_k$ to find the number of connected
non-isomorphic subgraphs of the tree $\T{k}$ that gives us the
number of $\dc$-classes of semigroup $\ipwr{k}$.

We use induction on $k$ to calculate the cardinality of the set of
graphs-representatives.

If $k=1$, then $\big|GRep_1 \big|=d+1=\binom{d+1}{d}=P(1)$.

Assume that $N=P^{k-1}(1)$  is the cardinality of the set of
graphs-representatives $GRep_{k-1}$ of $(k-1)$-level tree.

Each vertex of the first level of the tree $\T{k}$ is the root
vertex of $(k-1)$-level tree that is isomorphic to a certain graph
from $GRep_{k-1}$. Assume that all vertices of the first level are
labelled with integers from 1 to $l$, $l \leq d$. As set $GRep_k$
contains no equal graphs, then corresponding sequences are all
different. Hence, there exists one-to-one correspondence between
set $\{1,2, \ldots, l\}$ and set $\{0,1, \ldots, N-1\}$. Consider
all non-decreasing functions $f:\{1,2, \ldots, l\}\ra
\{0,1,\ldots, N-1\}$. The number of all such functions is equal to
the cardinality of the set $GRep_k$. Define  $x_0=f(1),
x_1=f(2)-f(1), \ldots, x_{l-1}=f(l)-f(l-1),x_l=N-1-f(l)$. Then
$f(k)=x_0+x_1+x_2+ \ldots +x_k$. Since $f$ is non-decreasing, then
for all $i=1, \ldots, l$ $x_i \geq 0$.So the number of
non-decreasing functions is equal to the number of integer
solutions of the equation $N-1=x_0+x_1+ \ldots + x_l$ for $l=1,
\dots, d$. Thus, we get the number of $\dc$-classes of semigroup
$\ipwr{k}$:
$$\sum_{l=0}^d \binom{N+l-1}{l}=
\binom{N+d}{d}=P(N)=P(P^{k-1}(1))=P^k(1)$$

\end{proof}
Let $\g$ be a subtree of the tree $\T{k}$,  $St_{\T{k}}(\g)$ be
the stabilizer of the subtree $\g$,  $Fix_{\T{k}}(\g)$ be the
fixator of the subtree $\g$ and let  $D_{\g}$ be $\dc$-class such
that for any $\s \in D_{\g}$ $\dom(\s) \cong \g$. Let $\{\g_1,
\ldots, \g_i\}$ be the set of all pairwise non-isomorphic subtrees
of $\g$ with root vertices in the first level of $\g$. Let
$\alpha_j$ be the number of isomorphic to $\g_j$ subtrees of $\g$
with root vertices in the first level of $\g$, $j=1, \ldots, i$.
The type of $\g$ is a set $\{(\g_1,\alpha_1), (\g_2,\alpha_2),
\ldots, (\g_i,\alpha_i)\}$ such that disjoint union of vertices
sets of all subtrees and root vertex gives vertices of $\g$.
Notice that $\sum_{j=1}^i \alpha_j=l$, where $l$ is the degree of
root vertex of $\g$.

\begin{prop}
Let $\g$ be a subtree of type $\{(\g_1,\alpha_1), (\g_2,\alpha_2),
\ldots, (\g_i,\alpha_i)\}$ of the tree $\T{k}$ and degree of the
root vertex of $\g$ be $l$, $l \leq d$. Then
\begin{enumerate}

\item $|\aut \g|=\prod_{j=1}^{i}(\alpha_j)!|\aut
(\g_j)|^{\alpha_j}, l \leq d$

\item $|St_{tx{k}}(\g)|=(d-l)!|\aut \T{k-1}|^{d-l} \prod_{j=1}^i
(\alpha_j)!|St_{\T{k-1}}(\g_j)|^{\alpha_j},$

\item $|Fix_{\T{k}}(\g)=(d-l)!|\aut \T{k-1}|^{d-l}\prod_{j=1}^i
|Fix_{\T{k-1}}(\g_j)|^{\alpha_j}.$
\end{enumerate}
\end{prop}
\begin{proof}
\begin{enumerate}
\item Prove the proposition by induction on $k$. Let $\g$ be
one-level tree and root vertex degree be $l \leq d$, then
$|\aut\g|=l!$.

Assume we know the orders of groups $\aut \g_j$ for all $j=1,
\ldots, i$. Find the order of $\aut \g$. Degrees of root vertices
of isomorphic trees are equal. Let them be $l$. It is clear that
types of all trees isomorphic to $\g$ are equal up to the
permutation of items. Thus only permutation of the first level
vertices, and consequently permutation of subtrees of $\g$,
distinguishes graph $\g$ from isomorphic one. All the vertices of
the first level may permute, but with several restrictions,
namely, roots of non-isomorphic subtrees stay roots of
non-isomorphic subtrees. Since the orders of $\aut \g_j$ for all
$j=1, \ldots, i$ are known, we can derive the order of $\aut \g$:
$$|\aut \g|= \prod_{j=1}^i (\alpha_j)!|Aut \g_j|^{\alpha_j}.$$

\item The proof is analogous to the proof of the previous
statement.

Consider the stabilizer of subtree $\g$ in the automorphisms group
of the rooted tree $\T{1}$. Let the degree of the root of $\g$ be
$l$. Then it is obvious  that $|St_{\T{1}}(\g)|=l!(d-l)!$.

Assume now that we know the order of $St_{\T{k-1}}(\g)$. Let the
degree of the root of $\g$ be $l$. Then $(d-l)$ vertices of the
first level of $\g$ may permute and each of them is the root of
$(k-1)$-level tree $\T{k-1}$. Among $l$ vertices, as in proof of
the previous statement, distinguish only vertices that are roots
of isomorphic subtrees. Then
$$|St_{\T{k}}(\g)|=(d-l)!|\aut \T{k-1}|^{d-i} \prod_{j=1}^i
(\alpha_j)!|St_{\T{k-1}}|^{\alpha_j}(\g_j).$$

\item Taking into account that fixator of the subtree doesn't
allow vertices permutation of this subtree, the proof is analogous
to the proof of point 2.

\end{enumerate}
\end{proof}

\begin{prop}
The cardinality of the set of idempotents $E(D_\g)$ of class
 $D_\g$ equals
$$|E(D_\g)|=\frac{(d!)^{\frac{1-d^k}{1-d}}}{|St_{\T{k}}(\g)}|.$$
\end{prop}
\textit{Proof} follows from one-to-one correspondence between the
set of ranges of idempotents of $D_\g$ and the set $\aut
\T{k}/St(\g)$, and $|\aut \T{k}|=(d!)^{\frac{1-d^k}{1-d}}$.\hfill
$\Box$
\begin{cor} \label{cor:cor1}
The number of $\rc$-classes and the number of ${\lc}$-classes
containing in $\dc$-class $D_\g$ is equal to
$$\frac{(d!)^{\frac{1-d^k}{1-d}}}{|St_{\T{k}}(\g)}|.$$
\end{cor}
\textit{Proof} follows from the fact that in inverse semigroup
every $\lc$-class and every $\rc$-class contains exactly one
idempotent. \hfill $\Box$
\begin{cor} \label{cor:cor2}
The cardinality of $\hc$-class containing in $\dc$-class $D_\g$ is
equal to $|\aut \g|$.
\end{cor}
\begin{proof}
Let $\s,\tau \in \pa \T{k}$. Then $\s \hc \tau$ if and only if
$\dom(\s)=\dom(\tau)$ and $\ran(\s)=\ran(\tau)$. The statement is
now obvious.
\end{proof}
\begin{cor}
$|D_{\g}| =|E(D_\g)|^2|Aut \g|$.
\end{cor}
\textit{Proof} follows from corolaries \ref{cor:cor1} and
\ref{cor:cor2}. \hfill $\Box$

\end{document}